\documentclass[reqno,a4paper, 10pt]{amsart}
\usepackage{amsfonts}
\usepackage{amssymb, amsmath,amsthm,enumerate}
\usepackage[mathscr]{eucal}
\usepackage{url}

\theoremstyle{plain}
\newtheorem{thm}{Theorem}
\newtheorem{cor}[thm]{Corollary}

\newtheorem*{op}{Hypothesis}
\newtheorem*{GNCT}{The Guthrie-Nymann Classification Theorem}
\newtheorem*{OP}{Open Problem}

\begin{document}
\title{Achievable Cantorvals almost without reversed Kakeya conditions}
\author[F. Prus-Wi\'{s}niowski and J. Ptak]{ Franciszek Prus-Wi\'{s}niowski and Jolanta Ptak}

\newcommand{\eacr}{\newline\indent}

\address{\llap{*\,}Franciszek Prus-Wi\'{s}niowski\eacr
Instytut Matematyki\eacr
Uniwersytet Szczeci\'{n}ski\eacr
ul. Wielkopolska 15\eacr
PL-70-453 Szczecin\eacr
Poland\acr ORCID 0000-0002-0275-6122}
\email{franciszek.prus-wisniowski@usz.edu.pl}

\address{\llap{*\,}Jolanta Ptak \eacr
P\c{e}zino 37 \eacr
73-131 P\c{e}zino\eacr
Poland\eacr
 ORCID  0000-0003-1464-060X}
\email{jola19941504@wp.pl}


\date{\today}

\newcommand{\acr}{\newline\indent}

\begin{abstract}
Examples of achievable Cantorvals are constructed with reversed Kakeya conditions only on a set of asymptotic density zero which answers in positive the Problem 5.2 from \cite{MM}. Additionally,  the Lebesgue measure of the boundaries of these Cantorvals is found to be zero which does not answer the still open problem of existence of achievable Cantorvals with boundaries of positive measure.
\end{abstract}

\maketitle

\section{Introduction}
The investigation of the topological nature of sets of all subsums of  absolutely convergent real series began over 100 years ago with the paper \cite{Kakeya} by Soichi Kakeya (see also \cite{Kakeya2}). He discovered that the set of subsums of a convergent infinite series of positive terms always is compact and perfect. Kakeya noticed that there are at least three possibilities: the set of subsums can be a finite set, a \textsl{multi--interval set} (that is, a union of a finite family of closed and bounded intervals) or a Cantor set (that is, a set homeomorphic to the classic ternary Cantor set). Kakeya's results were rediscovered twice: by H. Hornig in 1941 \cite{H} and by P. Kesava Menon in 1948 \cite{Menon}, but an important discovery was made only in 1980 by A.D. Weinstein and B.D. Shapiro. They announced in \cite{WS} an example of an absolutely convergent series with the set of subsums being a Cantorval (that is, a bounded, regularly closed subset of $\mathbb R$ with boundary being a Cantor set). Actually, the name Cantorval originated in \cite{MO} and various topological characterizations and other properties of Cantorvals were discussed in \cite{MO}, \cite{BFPW1}, \cite{MNP} and \cite{BKP}. All Cantorvals are homeomorphic to each other and a set homeomorphic to a Cantorval is a Cantorval. The complete classification of sets of subsums of absolutely convergent series was established by J.A. Guthrie and J.E. Nymann in \cite{GN88} (see also \cite{NS} for an essential amendment of the initial proof). Before we state it, we need a few basic definitions.

Throughout this note, we will consider only non-increasing sequences $(a_n)$ of non-negative terms such that $\sum a_n$ converges. The \textsl{achievement set} of the sequence $(a_n)$ (or, equivalently, the \textsl{set of subsums} of the series $\sum a_n$) is defined by
$$
E\ =\ E(a_n)\ :=\ \left\{x\in\mathbb R:\ \exists\ A\subset\mathbb N\quad\ x=\sum_{n\in A}a_n\,\right\}.
$$
A set $B\subset\mathbb R$ is said to be \textsl{achievable} is there is a sequence $(a_n)$ such that $B=E(a_n)$. Given any positive integer $k$, we define the \textsl{set of $k$-initial subsums} of $(a_n)$ by
$$
F_k\ =\ F_k(a_n)\ :=\ \left\{x\in\mathbb R:\ \exists\ A\subset\{1,2,\ldots,k\,\} \qquad\ x=\sum_{n\in A}a_n\,\right\}.
$$
Additionaly, we set $F_0:=\{0\}$ independently of the sequence $(a_n)$.
Throughout this note, the $n$-th remainder of a series $\sum a_i$ will be denoted by $r_n:=\sum_{i>n}a_i$. In particular, $r_0$ stands for the sum of the series. Connectivity components of $E$ that are not singletons are called \textsl{$E$-intervals}. Connectivity components of $[0,r_0]\setminus E$ are called \textsl{$E$-gaps}. We are now ready to formulate the complete topological classification of achievement sets.

\begin{GNCT}
The achievement set of any absolutely convergent real series is of one of the following four types:
\begin{itemize}
\item[(i)] a finite set;
\item[(ii)] a multi-interval set;
\item[(iii)] a Cantor set;
\item[(iv)] a Cantorval.
\end{itemize}
\end{GNCT}

Clearly, any two sets of the type (i) or any two sets of the type (ii) are homeomorphic if and only if they have the same number of connectivity components. Since the topological definition of a Cantorval is not very intuitive, it is useful to keep in mind the fundamental example of a Cantorval which is obtained by adding to the classic ternary Cantor set $C$ all open invervals removed from $[0,1]$ in the odd steps of the standard geometric construction of C \cite[pp.325-326]{GN88}. It is worth mentioning that this model Cantorval is not achievable (\cite{BFGSW}, \cite{BFHLP}). After the publication of the Guthrie-Nymann Classification Theorem, the investigation of achievement sets slowly gained momentum. Among others, achievement sets served as a counterexample (see \cite{Sannami92} and \cite{PWT18}) to the Palis hypothesis \cite{Palis87} that the arithmetic sum (or difference) of two Cantor sets, both with Lebesgue measure zero, is either of Lebesgue measure zero or it contains an interval. Palis hypothesis came from the theory of dynamical systems and added much interest to the investigation of algebraic sums of Cantor sets which is a thriving field of research (see, for example, \cite{PY97}, \cite{E07}, \cite{Pourba18}, \cite{T19}, \cite{MNP}). Achievement sets, due to their rigid nature, are a relatively small family of compact sets in the real line as can be seen from the fact that Palis hypothesis is generically true for dynamically defined Cantor sets \cite{MY01}. The elementary nature of achievement sets is often misleading because many seemingly straightforward facts require quite complicated proofs or constructions. In our opinion, the most challenging open problem related to achievement sets is to find new and easy to use sufficient conditions on $(a_n)$ that would guarantee that the set of subsums is a Cantor set.

Analytic characterization of absolutely convergent series having a finite set of subsums is obvious: $E(a_n)$ is a finite set if and only if $a_n=0$ for all sufficiently large indices $n$. The remaining three cases of possible topological types are less obvious. It was Kakeya who obtained first results in the direction. As a probing tool he considered inequalities between the terms of a series and the corresponding remainders. The equality $E(a_n)=\sum_{a_n<0}a_n\,+\,E(|a_n|)$ (see \cite[p.316]{H} or \cite[Lemma 1.2]{Jones}) shows that the achievement set of an absolutely convergent series $\sum a_n$ always is a translate of an achievement set of a series of nonnegative terms. Since we already know the nature of series with almost all terms equal to 0 and since neither permutations of terms nor removal of any  number of zero terms change the achievement set, it suffices to restrict our attention to convergent series of positive and monotone terms as Kakeya did in his investigations. Thus, from now on till the end of this note we will work under this restriction. Given an $n\in\mathbb N$, we will say that a series $\sum a_n$ satisfies the \textsl{$n$-th Kakeya condition} if $a_n>r_n$. We will also say that it satisfies the \textsl{$n$-th reversed Kakeya condition} if $a_n\le r_n$. For that reason, the following symbols will be handy:
$$
K(a_n)\,:=\,\bigl\{n\in\mathbb N:\ a_n>r_n\,\bigr\} \qquad\text{and}\qquad K^c(a_n)\,:=\,\bigl\{n\in\mathbb N:\ a_n\le r_n\,\bigr\}.
$$
These sets complement each other to $\mathbb N$. Each subset of $\mathbb N$ falls into exactly one of the following three categories: finite sets,  sets with finite complement and, finally, infinite sets with infinite complement. The first two of these categories showed up among Kakeya's results on relationship between the terms of a series and the nature of the sets of subsums.

\begin{thm}
\label{t1}
$\text{card}\,K(a_n)<\infty$ if and only if $E(a_n)$ is a multi-interval set.
\end{thm}

\begin{thm}
\label{t2}
If $\text{card}\,K^c(a_n)<\infty$, then $E(a_n)$ is a Cantor set.
\end{thm}

Of course, Theorem \ref{t2} can be rephrased as: if the set $K(a_n)$ has a finite complement, then $E(a_n)$ is a Cantor set. The implication cannot be reversed which can be seen from the series $\sum b_n$ with $b_{2k}=b_{2k-1}=\frac1{4^k}$ for all $k$. The easiest way to observe that  $E(b_n)$ is a Cantor set, is to apply  Theorem 16 from \cite{BFPW2}. Thanks to  Guthrie-Nymann Classification Theorem, the above theorems imply the following simple corollary.

\begin{cor}
\label{c3}
If $E(a_n)$ is a Cantorval, then $\text{card}\,K(a_n)=\,\text{card}\,K^c(a_n)=\infty$.
\end{cor}

The above observations determine uniquely the topological type of $E(a_n)$ if $K(a_n)$ is a finite set or a set with finite complement. Kakeya believed that \linebreak $\text card\,K(a_n) = \infty$ implies that $E(a_n)$ is a Cantor set. In \cite{Kakeya2} he wrote sincerely: \textit{ that the relation $a_n\le r_n$ fails only for an infinite number of values of $n$, seems to be the necessary and sufficient condition that the set $E(a_n)$ should be nowhere dense; but I have no proof of it}. The Guthrie-Nymann Cantorval $E(c_n)$, where $c_{2n-1}=\frac3{4^n}$ and $c_{2n}=\frac2{4^n}$, is a counterexample  to the hypothesis, since $K(c_n)=2\mathbb N$ (see \cite{GN88}, \cite{BPW}). The same example combined with the earlier one of $E(b_n)$ shows that the set of Kakeya conditions does not determine the topological type of achivement sets uniquely, because $K(c_n)=K(b_n)$. However, there is a pinch of truth in Kakeya's prediction, since every infinite set $K\subset\mathbb N$ with infinite complement is the set of Kakeya conditions for a suitable achievable Cantor set. Namely, three years ago, J. Marchwicki and P. Miska proved the following  non-trivial result in \cite{MM}. A simpler proof of it can be found in \cite{MPP}, but at the cost of losing the uniqueness of subsums.

\begin{thm}
\label{t4}
For any $K\subset\mathbb N$ with $\text{card}\,K=\text{card}\,K^c=\infty$, there is a series $\sum a_n$ such that $K(a_n)=K$ and $E(a_n)$ is a Cantor set.
\end{thm}

They investigated also if the Cantor set in the conclusion of Theorem \ref{t4} can be replaced by a Cantorval. The problem is not easy by all means. In all known examples of achievable Cantorvals, at least half (in the sense of asymptotic density) of the Kakeya conditions are reversed. For example, the historically first example of an achievable Cantorval was given by the Weinstein-Shapiro series $\sum d_n$ with $d_n=\frac3{10}\cdot\frac{9-m}{10^k}$ where $(k,m)$ is the unique pair of positive integers in $\mathbb N\times\{1,2,3,4,5\}$ such that $n=5(k-1)+m$ \cite{WS}. Since $K(d_n)=\{n\in\mathbb N:\ n=0(\text{mod}\,5)\,\}$, we have $d(K^c(d_n))=\frac45$. When we look at the Guthrie-Nymann series $\sum c_n$, then we get $d(K^c(c_n))=\frac12$.
Thus, Marchwicki and Miska investigated the auxilliary question of if it is possible to obtain a Cantorval from a sequence with  relatively fewer reversed Kakeya conditions and they obtained the following result in terms of asymptotic density. Let us recall its definition first. Given a set $C\subset \mathbb N$ and an $n\in\mathbb N$, we define $C_{\le n}:=\{k\in C:\ k\le n\}$. Then the asymptotic density of a set $A\subset\mathbb N$ is defined by
$$
dA\ :=\ \lim_{n\to\infty} \frac{\text{card}\,A_{\le n}}n.
$$
Replacing the limit by  the lower or the upper one, we get the lower and the upper asymptotic densities: $\underline{d}A$ and $\overline{d}A$.

\begin{thm}
\label{t5}
For any $0<\alpha\le\beta\le 1$ there is $\sum a_n$ such that $E(a_n)$ is a Cantorval and $\underline{d}K^c(a_n)=\alpha$ and $\overline{d}K^c(a_n)=\beta$.
\end{thm}

Additionally, they asked directly if it is possible to construct a series $\sum a_n$ such that $E(a_n)$ is a Cantorval and $dK^c(a_n)=0$ \cite[Problem 5.2]{MM}.

\section{The main result}

 We are going to give a positive answer to the question in this note, but we have to describe a tool necessary for the proof first.

Each multi-interval set $W$  is the union of infinitely many different finite families of closed and bounded intervals. However, one of those families stands out and is uniquely determined by $W$, namely, the family of all connectivity components of $W$. Thus, unless specified otherwise, by writing $W=\bigcup_{i=1}^nP_i$, we mean that $\{P_i: \ 1\le i\le n\,\}$ is the family of connectivity components of $W$. Given a multi-interval set $W=\bigcup_{i=1}^nP_i$, we define
$$
||W||\ :=\ \max_{1\le i\le n}|P_i|.
$$
If $(W_n)_{n\in\mathbb N}$ is a descending sequence of multi-interval sets, then $(||W_n||)_{n\in\mathbb N}$ is a non-increasing sequence bounded from below  and hence it converges to a non-negative number.
Even more, $\bigcap_nW_n$ contains an interval if and only if $\lim_{n\to\infty}||W_n||>0$ \cite[Lemma 2.7]{MNP}.

We need one more important definition. Given an $\epsilon>0$, a finite nonempty set $B\subset\mathbb R$ will be called \textsl{$\epsilon$-close} either if $\text{card}\,B=1$ or if distance between any two consecutive elements of $B$ (in the natural order inherited from $\mathbb R$) does not exceed $\epsilon$. An $\epsilon$-close subset $C$ of a finite set $B$ will be called \textsl{maximal} if it is maximal with repect to inclusion. The set of all  maximal $\epsilon$-close subsets of $B$ will be denoted by $\mathcal{M}_\epsilon(B)$. Every nonempty finite subset of real numbers has a unique decomposition into finitely many disjoint maximal $\epsilon$-close subsets. Given a nonempty finite set $C$, we define the \textsl{stretch} of the set by $S(C):= \max C-\min C$. Finally, given a finite nonempty set $B\subset\mathbb R$, we define
$$
\Delta_\epsilon B\ :=\ \max\bigl\{S(C):\ C\in \mathcal{M}_\epsilon(B)\,\bigr\}.
$$
The above definition is useful in characterizing series with nonempty interior of their achievement sets. The following theorem comes from \cite{MNP}.

\begin{thm}
\label{t6}
Let $\sum a_n$ be a convergent series of non-negative and non-increasing terms. Then the following conditions are equivalent:
\begin{itemize}
\item[(i)] the achievement set $E(a_n)$ contains an interval;
\item[(ii)] \ $\lim_{n\to\infty}\Delta_{r_n}F_n\,>\,0$;
\item[(iii)] \ $\lim_{k\to\infty}\Delta_{r_{n_k}}F_{n_k}\,>\,0$ for some increasing sequence $(n_k)$ of indices.
\end{itemize}
\end{thm}

We are now ready to formulate and prove our new result.

\begin{thm}
\label{t7}
For every sequence $(m_n)$ of positive integers convergent to $\infty$ there is an achievable Cantorval $E(a_n)$ such that
$$
\lim_{n\to\infty}\,\frac{\text{card}\,\{i\le n:\ a_i\le r_i\,\}}{m_n}\ \ =\ \ 0.
$$
\end{thm}
In particular, choosing $m_n:=n$, we get the classic asymptotic density and the positive answer to Problem 5.2 from \cite{MM}.

\begin{cor}
\label{c8}
There is an achievable Cantorval $E(a_n)$ such that the set $\{n:\ a_n\le r_n\,\}$ has asymptotic density 0.
\end{cor}

\begin{proof}[Proof of the Theorem \ref{t7}]
We start by describing a general scheme of constructing the desired sequence $(a_i)_{i\in\mathbb N}$. It combines the clever groups of terms introduced in the proof of Thm. 3.5 from \cite{MM} with the idea of multipliers $q_k$ for scaling of particular groups of terms used in the definition of generalized Ferens series in \cite{MNP}. Given a positive integer $n$, define
$$
b_1^n:=2^{n+1},\ b_2^n:=2^n+1\quad \text{and} \qquad b_j^n:=2^{n+3-j}\ \ \text{for $3\le j\le n+2$.}
$$
Then the set of all subsums of the terms $b_j^n$, $j=1,\,2,\,\ldots,\,n+2$, is exactly
\begin{multline*}
\left\{2j-2: \ j\in\{1,2,\ldots,2^{n-1}\}\right\}\\ \cup\ \left\{j:\ 2^n\le j\le 4\cdot2^n-1\,\right\}\\ \cup\ \left\{4\cdot2^n-1+2j:\ j\in\{1,2,\ldots,2^{n-1}\}\right\}.
\end{multline*}

We are ready to define the desired sequence $(a_n)_{n\in\mathbb N}$. We start by putting $N_0:=0$ and choosing a sequence $(N_k)_{k\in\mathbb N}$ of positive integers such that
\begin{equation}
\label{jeden}
\forall k\in\mathbb N \quad\forall n>N_k\qquad m_n\,\ge\,(k+1)^2
\end{equation}
and
\begin{equation}
\label{dwa}
\forall k\in\mathbb N\qquad N_k\,-\,N_{k-1}\ \ge\ 3.
\end{equation}
Next define $n_k:=N_k-N_{k-1}-2$ for all $k\in\mathbb N$, and $q_1:=1$ and
\begin{equation}
\label{trzy}
q_{k+1}\ :=\ \frac 1{3\cdot2^{n_{k+1}}}\,q_k \qquad \text{for $k\in\mathbb N$.}
\end{equation}
In particular, all $n_k$'s are positive integers because of \eqref{dwa}. Moreover,
\begin{equation}
\label{cztery}
\forall k\in\mathbb N\qquad q_{k+1}\ \le\ \frac{q_k}6.
\end{equation}
Then we complete the definition by setting $a_{N_{k-1}+i}:=b^{n_k}_iq_k$ for $k\in\mathbb N$ and $i\in\{1,2,\ldots, n_k+2\}$. Thus our sequence $(a_n)$ consists of groups of terms. The $k$-th group is built according to the first paragraph of the proof with $n=n_k$ and then multiplied by $q_k$. The first term of the $k$-th group has index $N_{k-1}+1$ and the last term has index $N_k$, and the group consists of $n_k+2$ terms.

We will show now that  the achievement set $E(a_n)$ is a Cantorval with the property claimed in the thesis. First, observe that the sequence $(a_i)$ decreases. For $i\ne N_k$, $k\in\mathbb N$, it is evident, because of the definition of groups of terms. If $k\in\mathbb N$, then
$$
a_{N_k}\ =\ 2q_k\ \overset{\eqref{trzy}}{=}\ 3\cdot2^{n_{k+1}+1}q_{k+1}\ >\ 2^{n_{k+1}}q_{k+1}\ =\ a_{N_k+1}.
$$
Next, we are going to determine the set $K(a_i)$ (or, equivalently, $K^c(a_i)$). Given a $k\in\mathbb N$, we have
\begin{align*}
r_{N_k}\ =\ \sum_{i=N_k+1}^\infty a_i\ &=\ \sum_{j=k}^\infty\,\sum_{i=N_j+1}^{N_{j+1}}a_i\ =\ \sum_{j=k}^\infty (5\cdot2^{n_{j+1}}-1)q_{j+1}\\[.1in]
&<\ \sum_{j=k}^\infty5\cdot2^{n_{j+1}}q_{j+1}\ \overset{\eqref{trzy},\eqref{cztery}}{\le}\ \frac53\sum_{j=k}^\infty \frac1{6^{j-k}}q_k\,=\,2q_k\,=\,a_{N_k}
\end{align*}
and hence the inequalities $a_i>r_i$ hold for all $i=N_{k-1}+3,\,N_{k-1}+4,\ldots, N_k$, $k\in\mathbb N$. It is easy to see directly from the definition of the sequence $(a_i)$ that $a_{N_{k-1}+2}\le r_{N_{k-1}+2}$ for all $k$. Similarly, we have $a_{N_{k-1}+1}\le r_{N_{k-1}+1}$ for all $k$, but here, if $n_k=1$ we need an additional estimate:
\begin{equation}
\label{szesc}
r_{N_k}\ >\ (5\cdot2^{n_{k+1}}-1)q_k\ \overset{\eqref{trzy}}{=}\ \tfrac53q_k\,-\,q_{k+1}\ \overset{\eqref{cztery}}{>}\ q_k.
\end{equation}
Thus,
\begin{equation}
\label{piec}
K^c(a_i)\ =\ \bigcup_{k=1}^\infty \bigl\{N_{k-1}+1,\, N_{k-1}+2\bigr\}.
\end{equation}
Since the set $K(a_i)$ is infinite, $E(a_i)$ has infinitely many gaps and the Guthrie-Nymann Classification Theorem will tell us that it is a Cantorval as soon as we show that $E(a_i)$ contains an interval. We will do it by using Prop. 2.6 from \cite{MNP}.

Define sets $C_k\,:=\,\bigl\{pq_k:\ p\in\{2^{n_k},\, 2^{n_k}+1,\ldots,\,2^{n_k+2}-1\}\,\bigr\}$ for $k\in\mathbb N$. All elements of $C_k$ are among subsums of the $k$-th subgroup of terms $a_{N_{k-1}+1},\,a_{N_{k-1}+2},\,\ldots,\,a_{N_k}$. Next, we are going to define a sequence of sets $(D_n)_{n\in\mathbb N}$  inductively. We start with $D_1:=C_1$. Clearly,
\begin{itemize}
\item[($\alpha_1$)] \ $D_1\,\subset F_{N_1}$;
\item[($\beta_1$)] \ the distance between any two consecutive elements of $D_1$ is $q_1$;
\item[($\gamma_1$)] \ $\min D_1\,=\,2^{n_1}q_1$ and $\max D_1\,=\,(2^{n_1+2}-1)q_1$.
\end{itemize}
Let $k$ be a positive integer such that the set $D_k$ has been defined and satisfies the properties:
\begin{itemize}
\item[($\alpha_k$)] \ $D_1\,\subset F_{N_k}$;
\item[($\beta_k$)] \ the distance between any two consecutive elements of $D_k$ is at most $q_k$;
\item[($\gamma_k$)] \ $\min D_k\,=\,\sum_{i=1}^k2^{n_i}q_i$ and $\max D_k\,=\,\sum_{i=1}^k(2^{n_i+2}-1)q_i$.
\end{itemize}
Define then
$$
D_{k+1}\ :=\ D_k\,+\,C_{k+1}.
$$
Clearly, $D_{k+1}$ satisfies $(\alpha_{k+1})$ and $(\gamma_{k+1})$. We are going to demonstrate that it satisfies $(\beta_{k+1})$ as well.

Let $x,\,y\in D_{k+1}$, $x<y$, be two consecutive points of $D_{k+1}$. Then $y=d+pq_{k+1}$ for some $d\in D_k$ and $p\in\bigl\{2^{n_{k+1}},\,2^{n_{k+1}}+1,\,\ldots,\,2^{n_{k+1}+2}-1\,\bigr\}$.

If $p>2^{n_{k+1}}$, then $\tilde{y}\le x<y$ for $\tilde{y}:=(p-1)q_{k+1}+d$ and hence $|x-y|\le q_{k+1}$.

If $p=2^{n_{k+1}}$, then $d>\min D_k$, because $y$ has a predecessor in $D_{k+1}$. Thus, $d$ has a predecessor in $D_k$. Denote it by $\tilde{d}$. By $(\beta_k)$ we get $d>\tilde{d}\le d-q_k$. Clearly, $\min(\tilde{d}+C_{k+1})<y$. On the other hand,
\begin{align*}
y\ &=\ d+2^{n_{k+1}}q_{k+1}\ =\ d-q_k+\bigl(q_k+2^{n_{k+1}}q_{k+1}\bigr)\ \overset{\eqref{trzy}}{=}\ d-q_k+2^{n_{k+1}+2}q_{k+1} \\[.1in]
&\overset{\text{$(\beta_k)$}}{\le}\ \tilde{d}+\bigl(2^{n_{k+1}+2}-1\bigr)q_{k+1}\,+\,q_{k+1}\ \le\ \max\bigl(\tilde{d}+C_{k+1}\bigr)\,+\,q_{k+1}.
\end{align*}
Since the distance between any consecutive elements of $\tilde{d}+C_k$ is equal to $q_{k+1}$, there is an $\hat{y}\in\tilde{d}+C_{k+1}$ such that $y-q_{k+1}\le\hat{y}<y$. Since $\hat{y}$ belongs to the set $D_{k+1}$, the property $(\beta_{k+1})$ has been proven and  the inductive definition of the sequence $(D_k)_{k\in\mathbb N}$ has been completed.

Since $r_{N_k}>q_k$ by \eqref{szesc} and $D_k\subset F_{N_k}$ by $(\alpha_k)$, we see that
$$
\Delta_{r_{N_k}}F_{N_k}\ \overset{\text{$(\gamma_k),\,(\beta_k)$}}\ \max D_k\ -\ \min D_k\ \overset{\text{$(\gamma_k)$}}{=}\ \sum_{i=1}^k(3\cdot2^{n_i}-1)q_i
$$
and hence $\lim_{k\to\infty}\Delta_{r_{N_k}}F_{N_k}\,>\,0$ which implies that $E(a_i)$ contains an interval by Proposition 2.6 from \cite{MNP}. Thus, $E(a_i)$ is a Cantorval.

Finally, for any $n\in\mathbb N$, there is a unique $k=k_n\in\mathbb N$ such that $N_{k-1}<n\le N_k$. Thus,
$$
\frac{\text{card}\{i\le n:\ a_i\le r_i\,\}}{m_n}\ \overset{\eqref{piec},\eqref{jeden}}{\le}\ \frac{2k}{k^2}\ \xrightarrow[n\to\infty]\ 0.
$$
\end{proof}

The strength of Theorem \ref{t7} can be better illustrated by looking at generalized densities.  Given any infinite subset $A\subset\mathbb N$, the $A$-density of a set $B\subset \mathbb N$ is defined by
$$
d_AB\ :=\ \lim_{n\to \infty}\,\frac{card\, B_{\le n}}{card\, A_{\le n}}
$$
if the limit exists. In particular, the classic asymptotic density is $d_{\mathbb N}$. Also, $d_AA=1$ for every infitite $A\subset \mathbb N$. The family $\mathcal{I}_A:=\{B\subset\mathbb N:\ d_AB=0\,\}$ is an ideal of sets (but not a $\sigma$-ideal). A set $C\subset \mathbb N$ is finite if and only if $d_AC=0$ for all infinite $A$. It follows that $\bigcap_A\mathcal{I}_A = \text{Fin}$. The Theorem \ref{t7} says that for every infinite $A\subset\mathbb N$, there is a series $\sum a_n$ such that $E(a_n)$ is a Cantorval and $d_AK^c(a_n)=0$, that is, none of the generalized densities provides a necessary condition in the form: if $E(a_n)$ is a Cantorval, then $d_AK^c(a_n)>0$. This makes us believe that the following open problem \cite[Problem 5.1]{MM} has a positive answer, but we have not proof of it.

\begin{op}
For each infinite $K\subset\mathbb N$ with infinite complement there is a  convergent series $\sum a_n$ of positive and monotone terms such that $E(a_n)$ is a Cantorval and $K^c(a_n)=K$.
\end{op}

Since the above hypothesis seems to be very resilient, it would be nice and enlightening to find at least a large family of subsets suitable for reversed Kakeya conditions for Cantorval producing series. For example, is it true that for every set $K\subset\mathbb N$ such that (a) $\text{card}\,K=\text{card}\, K^c=\infty$ and (b) if $n\in K$, then at least one of the numbers $n-1$ or $n+1$ belongs to K, there is a  convergent series $\sum a_n$ of positive and monotone terms such that $E(a_n)$ is a Cantorval and $K^c(a_n)=K$ ?

\section{The Lebesgue measure of the boundary}

One of unanswered questions of the theory of achievement sets is the question about the Lebesgue measure of the boundary of achievable Cantorvals. Can the measure be  positive? The boundary of a Cantroval is a Cantor set and we know from elementary measure theory that the Lebesgue measure of a Cantor set can be positive. Of course, it is very easy to produce an example of a Cantorval with boundary of positive measure. It suffices to take any central Cantor set of positive measure (see \cite[Problem 14b on p. 64]{R1988}) and fill in all open intervals removed in the odd-numbered steps of the standard geometric construction. Unfortunately, we still do not know whether any  of the Cantorvals obtained in this way is achievable. We doubt it, but we have no proof of it. On the other hand, in all cases when the measure of the boundary of an achievable Cantorval was computed, it was zero always \cite{BP}, \cite{BPW}, \cite{B}, \cite[Thm. 3.3]{MNP}. Thus, we are interested in the measure of the boundary of Cantorvals used in the proof of Theorem \ref{t7}.

Given any sequence $(n_k)$ of positive integers, we define a squence $(a_n)$ by
$$
a_i\ :=\ b_{i-\sum_{j=1}^{k-1}(n_j+2)}^{n_k}q_k \qquad \text{for $k\in\mathbb N$ and $\sum_{j=1}^{k-1}(n_j+2)\,<\,i\,\le\,\sum_{j=1}^k(n_j+2)$}
$$
where the terms $b_j^n$ have been defined in the first paragraph of the proof of the Theorem \ref{t7} and $q_k:=3^{1-k}\prod_{j=1}^{k-1}2^{-n_j}$. We use here the standard conventions that $\sum_{i=1}^0\ldots=0$ and
$\prod_{i=1}^0\ldots=1$. Then, as we know from the proof of Theorem \ref{t7}, $E(a_n)$ is a Cantorval and it will be called a \textsl{Marchwicki--Miska Cantorval}.

\begin{thm}
\label{t8}
The boundary of any Marchwicki-Miska Cantorval is of Lebesgue measure zero.
\end{thm}
\begin{proof}
Let $E=E(a_n)$ be a Marchwicki--Miska Cantorval. Denote the family of all $E$-intervals by $\mathcal{P}_E$. Then
$$
E\ =\ \bigsqcup_{P\in\mathcal{P}_E} \text{int}\,P\ \sqcup\ \text{Fr}\,E.
$$
Thus, by the countability of the family $\mathcal{P}_E$, we obtain
\begin{equation}
\label{b1}
\mu\,\text{Fr}\,E\ =\ \mu E\,-\,\sum_{P\in\mathcal{P}_E}\mu P.
 \end{equation}
In order to compute the difference, we need deep understanding of the geometry of $E$. For that purpose, we will use the $N_k$-th iterates $I_{N_k}$ of the achievement set where $N_k:=\sum_{i=1}^k(n_i+2)$ is the index of the last term in the $k$-th group of terms of the sequence $(a_n)$ \--- as in the proof of Theorem \ref{t7}. We know that $E=\bigcap_kI_{N_k}$ where  $I_{N_k}=\bigcup_{f\in F_{N_k}}[f,\,f+r_{N_k}] $ (see \cite[Fact 21.8]{BFPW1}). Defining for $k\in\mathbb N$
$$
M_k\ :=\ I_{N_k}\,\cap\,[0,\,r_{N_{k-1}}],
$$
we see that
\begin{multline*}
M_k\ =\bigsqcup_{j=1}^{2^{n_k-1}}\bigl[(2j-2)q_k,\,(2j-2)q_k+r_{N_k}\bigr]\ \sqcup\ \bigl[2^{n_k}q_k,\,(4\cdot2^{n_k}-1)q_k+r_{N_k}\bigr] \\
\sqcup\ \bigsqcup_{j=1}^{2^{n_k-1}}\bigl[(4\cdot2^{n_k}-1+2j)q_k,\,(4\cdot2^{n_k}-1+2j)q_k+r_{N_k}\bigr]
\end{multline*}
and
$$
I_{N_k}\ =\ \bigcup_{f\in F_{N_{k=1}}}(f\,+\,M_k).
$$
The iteration $I_{N_1}$ has $2^{n_1}$ gaps and $2^{n_1}+1$ intervals (that is, component itervals or, in other words, $I_{N_k}$-intervals). When passing from $I_{N_{k-1}}$ to $I_{N_k}$, $2^{n_k}$ new gaps appear in each interval of $I_{N_{k-1}}$. Denote the family of all new $I_{N_k}$-gaps by $\mathcal{G}_k$ for $k\in\mathbb N$. All elements of $\mathcal{G}_k$ are of the same length equal to $2q_k-r_{N_k}$. As a consequence, when passing from $I_{N_{k-1}}$ to $I_{N_k}$, each component of $I_{N_{k-1}}$ breaks up into $2^{n_k}+1$ intervals of $I_{N_k}$. More precisely, any breaking up interval $P$ of $I_{N_{k-1}}$ shrinks by $2^{n_k+1}q_k$ symmetrically, that is, by $2^{n_k}q_k$ on each side. Further, on each side of the shrinked $P$ arise $2^{n_k-1}$ new component intervals of $I_{N_k}$, all of length $r_{N_k}$ and with all gaps between any neighbouring  remnants of $P$ equal to $2q_k-r_{N_k}$.

For $k\in\mathbb N$, let us denote:
\begin{itemize}
\item \ $\mathcal{P}_k$\ \--- the family of all component intervals of $I_{N_k}$;
\item \ $\mathcal{P}'_k$\ \--- the family of component intervals of $I_{N_k}$ that are concentric with intervals of $I_{N_{k-1}}$. Informally, these are the symmetrically shrinked intervals of $\mathcal{P}_{k-1}$;
\item \ $\mathcal{P}''_k\ :=\ \mathcal{P}_k\setminus\mathcal{P}'_k$. Informally, they are the new intervals of $I_{N_k}$.
\end{itemize}
All intervals in $\mathcal{P}''_k$ have the same length equal to $r_{N_k}$ and there are $2^{n_k}\cdot \text{card}\,\mathcal{P}_{k-1}$ of them. Hence, by induction,
$$
\text{card}\,\mathcal{P}_k\ =\ \prod_{j=1}^k(2^{n_j}+1)
$$
and
\begin{equation}
\label{b2}
\text{card}\,\mathcal{P}''_k\ =\ 2^{n_k}\prod_{j=1}^{k-1}(2^{n_j}+1)\ =\ \text{card}\,\mathcal{G}_k.
\end{equation}
Since $E=[0,\,r_0]\setminus\bigcup_{k\in\mathbb N}\bigcup_{G\in\mathcal{G}_k}G$, we get
\begin{align}
\label{b3}
\notag \mu\, E\ &=\ r_0\ -\ \sum_{k=1}^\infty (2q_k-r_{N_k})\cdot\text{card},\mathcal{G}_k \\
&=\ r_0\,+\,\sum_{k=1}^\infty2^{n_k}r_{N_k}\prod_{j=1}^{k-1}(2^{n_j}+1)\ -\ \sum_{k=1}^\infty 2^{n_k+1}q_k\prod_{j=1}^{k-1}(2^{n_j}+1).
\end{align}

Let $C$ be the set of the centers of $E$-intervals and, for $k\in\mathbb N_0$, let $C_k$ be the set of centers of $I_{N_k}$-intervals. Then $C_k\subset C_{k+1}$ and $C\,=\,\bigcup_kC_k$. Given an $E$-interval $P$, we define $\eta(P)$ to be the smallest non-negative integer $k$ such that the center of $P$ belongs to $C_k$. For example, $\eta(P)=0$ if and only if $P$ is the central $E$-interval. Further, for $k\in\mathbb N$,
$$
\text{card}\,\bigl\{P\in\mathcal{P}_E:\ \eta(P)=k\,\bigr\}\ =\ \text{card}\,\mathcal{P}''_k.
$$
Given any $k\in\mathbb N_0$, all $E$-intervals $P$ with $\eta(P)=k$ are of the same length equal to $r_{N_k}\,-\,\sum_{i=k+1}^\infty2^{n_i+1}q_i$ and with his observation we are ready for the final computations:
\begin{align*}
\mu\,\text{Fr}\,E\ &\overset{\text{\eqref{b1}}}{=}\ \mu\,E\,-\,\sum_{k=0}^\infty\sum_{\substack{P\in\mathcal{P}_E\\ \eta(P)=k}}\mu\,P \\[.1in]
&=\ \mu\,E\,-\,\left(r_0\,-\,\sum_{i=1}^\infty2^{n_i+1}q_i\right)\ -\ \sum_{k=1}^\infty\text{card}\,\mathcal{P}''_k\cdot\left(r_{N_k}\,-\,\sum_{i=k+1}^\infty 2^{n_i+1}q_i\right) \\[.1in]
&\overset{\text{\eqref{b2},\eqref{b3}}}{=}\ \sum_{k=1}^\infty 2^{n_k+1}q_k\,+\,\sum_{k=1}^\infty\left(2^{n_k}\biggl(\prod_{j=1}^{k-1}(2^{n_j}+1)\biggr)\sum_{i=k+1}^\infty 2^{n_i+1}q_i\right) \,-\,\sum_{k=1}^\infty 2^{n_k+1}q_k\prod_{j=1}^{k-1}(2^{n_j}+1)\\
&=\  \sum_{k=1}^\infty 2^{n_k+1}q_k\,+\,\sum_{i=2}^\infty\left(2^{n_i+1}q_i\sum_{k=1}^{i-1}2^{n_k}\biggl(\prod_{j=1}^{k-1}(2^{n_j}+1)\biggr)\right)\,-\,\sum_{k=1}^\infty 2^{n_k+1}q_k\prod_{j=1}^{k-1}(2^{n_j}+1)\\
\intertext{(and now we interchange the symbols $i$ and $k$ in the second summation only)}
&=\ \sum_{k=1}^\infty 2^{n_k+1}q_k\,+\,\sum_{k=2}^\infty\left(2^{n_k+1}q_k\sum_{i=1}^{k-1}2^{n_i}\biggl(\prod_{j=1}^{i-1}(2^{n_j}+1)\biggr)\right)\,-\,\sum_{k=1}^\infty 2^{n_k+1}q_k\prod_{j=1}^{k-1}(2^{n_j}+1)\\
&=\ \sum_{k=2}^\infty 2^{n_k+1}q_k\left[1+\sum_{i=1}^{k-1}2^{n_i}\biggl(\prod_{j=1}^{i-1}(2^{n_j}+1)\biggr)-\prod_{j=1}^{k-1}(2^{n_j}+1)\right] \\
&=\ \sum_{k=2}^\infty 2^{n_k+1}q_k\cdot0\ =\ 0.
\end{align*}
\end{proof}

\begin{OP}
Is the boundary of any achievable Cantorval a set of Lebesgue measure zero?
\end{OP}

\bibliographystyle{amsplain}

\end{document}